\documentclass{article}
\usepackage{graphicx}
\usepackage{amsmath}
\usepackage{amsfonts}
\usepackage{amssymb}
\usepackage{graphicx}
\usepackage{tikz}
\usepackage{tikz-cd}

\usepackage{amsthm}
\usepackage{mathrsfs}

\newtheorem{theorem}{Theorem}
\newtheorem{lemma}{Lemma}

\newenvironment{AMS}{}{}
\newenvironment{keywords}{}{}

\newtheorem{defn}[equation]{Definition}
\newtheorem{prop}[equation]{Proposition}
\newtheorem{rem}[equation]{Remark}

 \title{ Multiplication Operator Semigroups on Banach lattice valued continuous function spaces }
\author{Tobi David Olabiyi\footnote{Stellenbosch University, South Africa.\\ 25175645@sun.ac.za/davidtobiolabiyi@gmail.com.\\
	The author was supported by the German Academic Exchange Service (DAAD).}
} 
\date{\today}

\newcommand{\Bm}[1]{ \boldsymbol{ #1 }}

\usepackage{xcolor}
\usepackage[linktoc=all, colorlinks=true]{hyperref} 
\hypersetup{
	linkcolor=red,
	citecolor=blue,
	urlcolor=teal,
}

\newcommand{\C}{\mathbb{C}}
\newcommand{\R}{\mathbb{R}}
\newcommand{\N}{\mathbb{N}}

\newcommand{\mI}{\mathcal{I}}

\newcommand{\sL}{\mathscr{L}}
\newcommand{\mZ}{\mathcal{Z}}

\newcommand{\mO}{\mathcal{O}}

\newcommand{\mM}{\mathcal{M}}

\newcommand{\mT}{\mathcal{T}}

\newcommand{\cc}{\circ}

\newcommand{\ov}[1]{ \overline{#1}}

\newcommand{\q}{\quad}
\newcommand{\ep}{\varepsilon}

\newcommand{\ty}{\infty} 
 
\newcommand{\cd}{\cdot}
\newcommand{\ti}{\times}

\newcommand{\ra}{\rightarrow}

\newcommand{\da}{\downarrow}
\newcommand{\lo}{\longrightarrow}

\newcommand{\sse}{\subseteq}

\newcommand{\bk}{\backslash}

\newcommand{\la}{\lambda}
\newcommand{\Ga}{\Gamma}
\newcommand{\ga}{\gamma}

\newcommand{\sig}{\sigma}

\newcommand{\de}{\delta}

\newcommand{\Om}{\Omega}

\newcommand{\mt}{\mapsto}
\newcommand{\mf}\mapsfrom {}

\newcommand{\lb}{\left\lbrace}
\newcommand{\rb}{\right\rbrace }
\newcommand{\bcc}{\begin{cases}}
	\newcommand{\ecc}{\end{cases} }
\newcommand{\bet}{\begin{tikzcd}}
	\newcommand{\eet}{\end{tikzcd} }
\newcommand{\bei}{\begin{itemize}}
	\newcommand{\eei}{\end{itemize}}
\newcommand{\ben}{\begin{enumerate}}
	\newcommand{\een}{\end{enumerate}}
\newcommand{\beq}{\begin{equation}}
	\newcommand{\eeq}{\end{equation}}
\newcommand{\bec}{\begin{center}}
	\newcommand{\eec}{\end{center}}
\newcommand{\beqn}{\begin{equation*}}
	\newcommand{\eeqn}{\end{equation*}}
\newcommand{\beqa}{\begin{eqnarray}}
	\newcommand{\eeqa}{\end{eqnarray}}

\newcommand{\ban}{\begin{align*}}
	\newcommand{\ean}{\end{align*}}

\newcommand{\beqan}{\[\begin{aligned}}
	\newcommand{\eeqan}{\end{aligned}\]}

\begin{document}
	\maketitle
	\begin{abstract}
		We introduce and characterize, on the Banach lattice valued continuous function space, multiplication operators generating strongly continuous multiplication operator semigroups. Our characterization is the generalization of  known results for the scalar-valued continuous  functions $C_0(\Om)$ vanishing at infinity, on a locally compact (Hausdorff) space $\Om$, to Banach lattice $C_0(\Om, E)$ of continuous  Banach lattice $E$-valued functions vanishing at infinity.
	\end{abstract}
	
	

	\begingroup
	\hypersetup{linkcolor=black}
	\tableofcontents
	\endgroup	
\section{Introduction}\label{sec:intro}

The study of multiplication operators and their operator semigroups have attracted a lot of attention, see for example \cite{Holderrieth91,van1993abstract,Heymann2013, hudzik2014multiplication,MerveSerkanEmrah2020}, majorly because multiplication operators arise naturally. Indeed, (i) multiplication operators can be considered, in a classical sense, as the infinite-dimensional generalization of diagonal matrices; and (ii)
there are several instances where certain kind of operators, e.g., normal operators on a separable infinite-dimensional Hilbert space  or differential operators on a scalar-valued function space via Fourier transforms, can be represented as multiplication operators.

These multiplication operators are also very much connected with certain kinds of multipliers, e.g., some operator multipliers, see \cite{Graser97}, on Banach space valued function spaces and the  Hadamard multipliers, see \cite{Anna2018}, on scalar-valued analytic function spaces, e.t.c.  In particular, in \cite{Graser97, ChristianRetha22}, the authors combine the study of extrapolation spaces with operator-valued multiplication operators. More so, these multiplication operators have  a relation with the  multiplier algebras, see \cite{TodorovTurowska2010}, in the setting of Banach algebras.

In \cite[Chapter I, section 3, p.19-29]{engel2006short}  and \cite[Chapter II, section 2, p.50]{engel2006short}, certain multiplication operators and their operator semigroups on concrete spaces $C_0(\Om)$ and $L^p(\Om, \mu),\ 1\le p<\ty$, where either $\Om$ is a locally compact (Hausdorff) space or $(\Om, \mu)$ is a (\textit{positive}) $\sig$-finite measure space respectively, are considered. We are in particular interested in the generalization of the results for the scalar-valued continuous  functions $C_0(\Om)$ vanishing at infinity to Banach lattice $C_0(\Om, E)$ of continuous  Banach lattice $E$-valued functions vanishing at infinity. 	
Generally, for a Banach lattice $E$, the lattice algebra
 \[\Bm{\mZ}(E):= \lb T\in \sL(E) : \exists\   \la> 0 \ \text{such that} \ \textbf{$|Tz|$} \le \la \textbf{$|z|$}\ \forall \ z\in E\rb\] 
of central operators, the so-called centre of $E$ equipped with the operator norm (equivalently, $||T||_{\sL(E)}= \inf\lb \la> 0 :\textbf{$|Tz|$} \le \la \textbf{$|z|$}\ \forall \ z\in E\rb $), is a commutative 1-Banach lattice algebra. We refer to \cite{wickstead17} for concepts of Banach lattice algebras, including examples. Moreover, $\Bm{\mZ}(E)$ is an AM-space with unit. Elements of this lattice algebra are called multiplication operators, for obvious reasons. Indeed, for concretes spaces $C_0(\Om)$ and $L^p(\Om, \mu),\ 1\le p\le\ty$, their centres can be identified uniquely with multiplications by functions in $C_b(\Om)$ and $L^\ty(\Om, \mu)$ respectively. And probably what is important to note here is that for a Banach lattice $E$, its centre $\Bm{\mZ}(E)$, with the uniform operator topology, is isometrically isomorphic to $C(Q)$ as commutative Banach lattice algebras for some compact (Hausdorff) space $Q$ (see also \cite[Section C-I.9, p.246]{arendt1986}). 
We let $\Bm{\mZ}(E)_\textbf{s}$ denotes $\Bm{\mZ}(E)$ equipped with the strong operator topology.

 \bigskip          

The major results of this paper are summarized in the following theorem.

\begin{theorem}\label{thm1}
	Let  $m_t\in C_b(\Om, \Bm{\mZ}(E)_\textbf{s})$ for each $t\ge0$ such that the induced multiplication semigroup $\mT(t)_{t\ge0}$ given by
\[ \mT(t) : C_0(\Om, E) \lo C_0(\Om, E);\ s(\cd)\mt \mT(t)s(\cd):=m_t(\cd)s(\cd)     \]
is a $C_0$-semigroup on $C_0(\Om, E)$. And denote its (infinitesimal) generator  on $C_0(\Om, E)$ by $(\mM, D(\mM))$. Furthermore, suppose that, for each $x\in\Om$, the induced pointwise multiplication semigroup $T_x(t)_{t\ge0}$ given by
\[T_x(t): E \lo E;\ z\mt T_x(t)z:=m_t(x)z    \]
is uniformly continuous on $E$.
Then we have the following.

\bei
\item[(i)] There exists a (generally unbounded) continuous function $\phi : \Om \lo  \Bm{\mZ}(E)_\textbf{s}$, where $(T_x(t))_{t\ge 0}= (e^{t\phi(x)}  )_{t\ge 0}$ for each $x\in\Om$ with
\[\sup_{x\in\Om}||e^{t_0\phi(x)}||_{\sL(E)}<\ty \ \text{ for some }\ t_0\in (0,1]  \]
such that $m_t(\cd)= e^{t\phi(\cd)}$ for all $t\ge0$. That is $\mT(t)_{t\ge0} =(\mT_\phi(t))_{t\ge 0}$, where $\mT_\phi(t)s(\cd):= e^{t\phi(\cd)}s(\cd)$ for all $s\in C_0(\Om, E)$ and $t\ge0$.

\item[(ii)]  $(\mM, D(\mM)) = (\mM_\phi,  D(\mM_\phi)  ) $, where $(\mM_\phi,  D(\mM_\phi)  )$ denotes the (generally unbounded)  multiplication operator on $C_0(\Om, E)$ induced by the continuous function $\phi$ and defined on its "maximal domain" $D(\mM_\phi)$.

\item[(iii)]   The multiplication $C_0$-semigroup $(\mT_\phi(t))_{t\ge 0}$ is uniformly continuous if and only if $\phi\in C_b(\Om, \Bm{\mZ}(E)_\textbf{s})$.

\eei	
	
\end{theorem}
	\bigskip

	The following is the organization of this paper. We start in section \ref{sec:mop}, concerning the statement (ii) of Theorem \ref{thm1}, by showing that the multiplication operator $(\mM_\phi,  D(\mM_\phi)  )$ induced by a continuous function $\phi : \Om \lo  \Bm{\mZ}(E)_\textbf{s}$ is generally an unbounded but closed and densely defined operator; which is bounded if and only if $\phi$ is bounded. Moreover, we characterize the spectrum and resolvent set of $\mM_\phi$ (see Proposition \ref{prop:mop1}). In section \ref{sec:msg}, we justify the condition appearing in the statement (i)  of Theorem \ref{thm1}, see Remark \ref{rem:msg1}, as what is necessary and sufficient for every  multiplication operator $(\mM_\phi,  D(\mM_\phi)  )$ generating a strongly continuous multiplication semigroup  $(\mT_\phi(t))_{t\ge 0}$ on $C_0(\Om, E)$; and that the semigroup is uniformly continuous if and only if $\phi$ and hence  $\mM_\phi$ is bounded (see Propositions \ref{prop:msg1} and Lemma \ref{lem:msg2}). We concluded this paper in section \ref{sec:exm} with some observations and concrete examples of such multiplication semigroups.

\section{Multiplication Operators on $C_0(\Om, E)$.}\label{sec:mop}

Given a locally compact (Hausdorff) space  $\Om$ and a Banach lattice $E$, we know that the centre $\Bm{\mZ} ( C_0(\Om, E))$ of the Banach lattice $C_0(\Om, E)$ is isometrically isomorphic to the space $C_b(\Om, \Bm{\mZ}(E)_\textbf{s})$ as commutative 1-Banach lattice algebras (see  \cite[Theorem 6.2, p.135]{ErcanWickstead96}). In his thesis, see \cite[Chapter 3, Proposition 3.6.6.1, p.113-116]{David2022}, the author proved that $C_b(\Om, \Bm{\mZ}(E)_\textbf{s})\cong \Ga_b(\Om, E_\textbf{z})$ where $ \Ga_b(\Om, E_\textbf{z})$ denotes the Banach lattice algebra of bounded continuous sections associated with the trivial bundle $E_\textbf{z}:=\Om \ti \Bm{\mZ}(E)_\textbf{s}$, by which we, in particualr, realize $\Bm{\mZ} ( C_0(\Om, E))\cong C_b(\Om, \Bm{\mZ}(E)_\textbf{s})\cong \Ga_b(\Om, E_\textbf{z})$ as isomorphism of AM m-lattice modules over the centre $ \Bm{\mZ}(C_0(\Om))\cong C_b(\Om)$.

Just a quick note: AM m-lattice modules over Banach lattice algebra $C_0(\Om)$ are also what we referred to as locally convex  $C_0(\Om)$-m-lattice modules or equivalently upper semicontinuous  $C_0(\Om)$-functions  m-lattice modules (see \cite[Chapter 2, Corollary 2.4.1.10, p.32]{David2022}). The author proved that, see \cite[Chapter 3, Proposition 3.5.0.9, p.85-88]{David2022}, these are precisely those Banach lattice modules over $C_0(\Om)$ which can uniquely be represented as Banach lattices of continuous sections of some topological bundles of Banach lattices over $\Om$. See also \cite[Corollary 7.28, p.78-79]{Gierz1982}.
\bigskip

Throughout we denote by $\mI\in \Bm{\mZ} ( C_0(\Om, E))$ the identity operator on $C_0(\Om, E)$, $I_E\in\Bm{\mZ}(E)_\textbf{s} $ the identity operator on $E$ and $e\in C_b(\Om, \Bm{\mZ}(E)_\textbf{s})$ the unit element of $C_b(\Om, \Bm{\mZ}(E)_\textbf{s})$. 
\bigskip 

With a continuous function $\phi : \Om \lo  \Bm{\mZ}(E)_\textbf{s}$ we associate a linear operator $\mM_\phi$ on $C_0(\Om, E)$ defined on its "maximal domain"  $D(\mM_\phi)$ in $C_0(\Om, E)$.  This definition is motivated by \cite[Definition 3.1, p.20]{engel2006short} and \cite[Definition 1.1, p.575]{Heymann2013}.

\begin{defn} (Multiplication and pointwise multiplication operators)\label{defn:mop1}
	The \textit{multiplication operator} $\mM_\phi $ induced on $C_0(\Om, E)$  by a  continuous function $\phi : \Om \lo  \Bm{\mZ}(E)_\textbf{s}$ is defined by
	\[\mM_\phi s=\phi \cc s,\text{ i.e., }\mM_\phi s(x):= \phi (x)s(x) \text{ for all } x\in \Om \text{ and domain}  \]
	\[ D(\mM_\phi):= \lb s\in C_0(\Om, E):\  \phi (\cd)s(\cd )\in C_0(\Om, E) \rb. \]

	In this context, we call the operators $\phi(x)$ with $x\in\Om$ the pointwise multiplication operators on $E$.
\end{defn}

We show in the following proposition certain relationships existing between the multiplication operator ($\mM_\phi,  D(\mM_\phi)$) and the continuous function $\phi$. This can be seen as a generalization of \cite[Chapter 1, Proposition 3.2, p.20-21]{engel2006short}.  See also  \cite[Proposition 2.3 \& Theorem 2.6, p.70]{Graser97}.

\begin{prop}\label{prop:mop1}
	Let $\mM_\phi : D(\mM_\phi)\sse C_0(\Om, E) \lo C_0(\Om, E); s \mt \phi\cc s$ be the multiplication operator induced by a continuous function $\phi : \Om \lo  \Bm{\mZ}(E)_\textbf{s}$. Then, we have the following.
	\bei	
	\item[(i)] The operator $(\mM_\phi,  D(\mM_\phi)  )$ is closed and densely defined.
	
	\item[(ii)] The operator $\mM_\phi$ is bounded (with $D(\mM_\phi) =C_0(\Om, E)$) if and only if $\phi$ is bounded, i.e., $\phi\in C_b(\Om, \Bm{\mZ}(E)_\textbf{s})$. In this case,  $\mM_\phi\in\Bm{\mZ} ( C_0(\Om, E))$ and one has
	$$||\mM_\phi||= ||\phi||= \sup_{x\in\Om}||\phi(x)||_{\sL(E)}.$$
	
	\item[(iii)] The operator $\mM_\phi$ has a bounded inverse if and only if $\phi$ has a bounded inverse $\phi^{-1}\in C_b(\Om, \Bm{\mZ}(E)_\textbf{s})$, i.e., $\phi(x)$ is invertible in $\Bm{\mZ}(E)_\textbf{s}$ for all $x\in\Om$, and the function $ x  \overset{\phi^{-1}}{\longmapsto}       \phi(x)^{-1}$ is in $C_b(\Om, \Bm{\mZ}(E)_\textbf{s})$.
	 In this case, one has
	\[  \mM_\phi^{-1}= \mM_{\phi^{-1} }.\]
	\item[(iv)] For the spectrum of $\mM_\phi$, we have that
		\[\sig(\mM_\phi)=  \Delta_\phi:=  \lb  \la\in\C :\la \in\bigcup_{x\in \Om}\sig(\phi(x)) \ \text{ or } \     \sup_{x\in\Om}|| R(\la, \phi(x))||_{\sL(E)}=\ty  \rb.\] 
	That is, the resolvent set of $\mM_\phi$ is given by 
	\[ \rho(\mM_\phi) = \Delta_u:= \lb \la\in\C : \la \in \bigcap_{x\in \Om} \rho(\phi(x)) \ \text{ and } \  \sup_{x\in\Om}|| R(\la, \phi(x))||_{\sL(E)}<\ty \rb.\]	
	\eei
\end{prop}

\begin{proof}
	\bei
	\item[(i)] It is clear that, the domain $D(\mM_\phi) $ always contain the space
	\[   C_c(\Om, E):=  \lb s\in C(\Om, E) : \text{supp }s \text{ is compact} \rb\] of all continuous functions $s : \Om \lo E$ having compact support
	\[\text{supp }s:=  \ov{ \lb x\in\Om : s(x)\ne 0\in E \rb}.  \]
	So, to prove that the operator $(\mM_\phi,  D(\mM_\phi)  )$  is densely defined, it suffices to show $C_c(\Om, E)$ is norm dense in $C_0(\Om, E).$ Indeed, if we take approximate identity $(e_j)_{j}$ of $C_0(\Om)$, and without loss of generality, we may assume that $e_j$ has compact support for each $j$, so that for every $s\in C_0(\Om, E)$, we have that $e_j\cd s\in C_c(\Om, E)$ for each $i$, and $\lim_{j}||s - e_j\cd s|| = 0$.

	Next, we show that $(\mM_\phi,  D(\mM_\phi)  )$ is a closed operator. Let $(s_n)_{n\in\N}\sse D(\mM_\phi)$ be a sequence converging to $s\in 	C_0(\Om, E)$ such that $\lim_{n\ra \ty} \phi \cc s_n =:s_0\in C_0(\Om, E)$ exists. Then, it must be that, for each $x\in \Om$, $\lim_{n\ra \ty}||s_n(x)-s(x)||_E=0$ and $\lim_{n\ra \ty}||\phi(x)s_n(x) - s_0(x)||_E=0$, from which it follows that $\lim_{n\ra \ty}||\phi(x)s_n(x)-\phi(x)s(x)||_E=0$. Thus, for each $x\in \Om$, $s_0(x)=\phi(x)s(x)$, so that $s\in D(\mM_\phi)$ and $s_0= \phi\cc s$.
	\item[(ii)] If $\phi\in C_b(\Om, \Bm{\mZ}(E)_\textbf{s})$, we have that
	\[  ||\mM_\phi s|| = \sup_{x\in\Om}||\phi(x)s(x)||_E \le ||\phi||||s||   \]       
	for every $s\in C_0(\Om, E)$; hence $\mM_\phi$ is bounded with $||\mM_\phi||\le ||\phi||$. And on the other hand, assume $\mM_\phi$ is bounded. Now, for each $x\in\Om$, let $f_x$ be a continuous function in $C_0(\Om)$ with compact support satisfying $||f_x||= 1 =f_x(x).$ And for (fixed) $z_0\in E$ with $||z_o||=1$, consider, for each $x\in\Om$, the continuous mapping $f_x\otimes z_0 : \Om \lo E ; y \mt (f_x\otimes z_0)(y):= f_x(y)z_0$. Then, we have that  $f_x\otimes z_0 \in C_c(\Om, E)$ with $||f_x\otimes z_0||=1$; and moreover,  for each $x\in\Om$, 
	\[||\phi(x)z_0||= ||\mM_\phi(f_x\otimes z_0  )(x)||\le ||\mM_\phi(f_x\otimes z_0  )||\le ||\mM_\phi||\]
	implies that $||\phi(x)z||\le ||\mM_\phi||$ for every $z\in E$ with $||z||=1$. And as such, we have that
	$||\phi(x)||\le ||\mM_\phi||$ for every $x\in \Om$; hence, $\phi\in C_b(\Om, \Bm{\mZ}(E)_\textbf{s})$ with $$||\phi||=\sup_{x\in\Om}||\phi(x)||_{\sL(E)}\le ||\mM_\phi||.$$
	The fact that $\mM_\phi$ is a multiplication operator on 	$C_0(\Om, E)$ follows, since the pointwise $(\phi(x)_{x\in\Om})$ are multiplication operators on $E$ which are continuously bounded on $\Om$. Indeed, for each $s\in C_0(\Om, E)$,
		\beqan
	|\mM_\phi s|(x)	 =    &  |\phi(x)s(x)| \\
	\le & ||\phi(x)||_{\sL(E)}|s(x)|\\
	\le & ||\phi|| \cd |s|(x)
	\eeqan
	for all $x\in\Om$, implies that $|\mM_\phi s|\le ||\phi||\cd |s|$. Thus $\mM_\phi\in\Bm{\mZ} ( C_0(\Om, E))$.

	\item[(iii)] Suppose $\phi$ has bounded inverse $\phi^{-1}\in C_b(\Om, \Bm{\mZ}(E)_\textbf{s})$. This implies that $\phi^{-1}\cc\phi =\phi\cc\phi^{-1} =e \in C_b(\Om, \Bm{\mZ}(E)_\textbf{s})$.  Then, the induced (bounded) multiplication operator $$\mM_{\phi^{-1}} : C_0(\Om, E) \lo C_0(\Om, E) ; s \mt \phi^{-1}\cc s$$ is the inverse of $\mM_\phi$, since  $\mM_{\phi^{-1}}\mM_\phi s = s$ for all $s\in D(\mM_\phi)$ and	
	$   \mM_\phi \mM_{\phi^{-1}} = \mM_e = \mI\in \Bm{\mZ} ( C_0(\Om, E))$. \bigskip

	Now, suppose $\mM_\phi$ has bounded inverse, say $\mT$, i.e., $\mT\mM_\phi s = s$ for all $s\in D(\mM_\phi)$ and $\mM_\phi\mT = \mM_e = \mI$. Then, we obtain
\[||s|| \le ||\mT|| \cd ||\mM_\phi s|| \q \text{ for all } s\in D(\mM_\phi),\]
	whence, the estimate
\[\de:= \frac{1}{||\mT||}\le \sup_{x\in\Om}||\phi(x)s(x)||_E \q \text{ for all } s\in D(\mM_\phi), \ ||s||=1.\]
	
	And suppose on the contrary, without loss of generality, we assume either $\inf_{x\in\Om}\inf_{||z||=1}||\phi(x)z||	<\frac{\de}{2}$ or $\inf_{x\in\Om} \frac{1}{||\phi(x)^{-1}||}<\frac{\de}{2}$; then we can find an open set $\mO \sse \Om$ such that either $\inf_{||z||=1}||\phi(x)z||<\frac{\de}{2}$ or $\frac{2}{\de}<\sup_{||z||=1}||\phi(x)^{-1}z||$ for all $x\in\mO$ respectively. Let $f_0$ be a continuous function in $C_0(\Om)$ such that $||f_0||=1$ and  $f_0(x)=0$ for all $x\in\Om\bk\mO$, and consider, for any $z\in E$ with $||z||=1$, the continuous mappings $f_0\otimes z : \Om \lo E ; x \mt (f_0\otimes z)(x):= f_0(x)z$; so that $f_0\otimes z \in D(\mM_\phi)$ and $||f_0\otimes z||=1$. Then, the assumptions imply either $\sup_{x\in\Om} ||\phi(x)(f_0\otimes z)(x)||\le\frac{\de}{2}$ or $\frac{2}{\de} >\sup_{x\in\Om}||\phi(x)^{-1}(f_0\otimes z)(x) ||   =\sup_{x\in\Om}||\mT(f_0\otimes z)(x)||  $ for some $z\in E$ respectively, which is contradicting the above estimate.

	Hence, we can conclude that each $\phi(x)$ is invertible for all $x\in\Om$ and $\sup_{x\in\Om}||\phi(x)^{-1}||= ||\mT || $, which would implies $\phi$ has bounded inverse $\phi^{-1}\in C_b(\Om, \Bm{\mZ}(E)_\textbf{s})$; so that $\mM_{\phi^{-1}}$ is the bounded multiplication operator on  $C_0(\Om, E)$ which coincides with the bounded inverse $\mT$.	
	
	Equivalently, $\Bm{\mZ} ( C_0(\Om, E))$ is an inverse-closed subalgebra of  $\sL( C_0(\Om, E)) $ and $ \mT(\phi \cc s) = s = \phi \cc \mT s$ for all $s\in D(\mM_\phi)$; it must be that the bounded inverse $\mT$ is a bounded multiplication operator on  $C_0(\Om, E)$, i.e., there exists $\psi\in C_b(\Om, \Bm{\mZ}(E)_\textbf{s})$ such that $\mM_\psi = \mT$. 	It, thus, follows that $\phi \cc \psi = e=\psi \cc\phi $, so that $\phi$ has bounded inverse $\phi^{-1}:=\psi$.
	
	\item[(iv)] By definition, one has $\la\in \sig(\mM_\phi)$ if and only if $\la - \mM_\phi = \mM_{\la -\phi}$ does not have bounded inverse. Thus, by (iii) above, it suffices to show that $\la -\phi$ has bounded inverse if and only if $\la\in\Delta_u$. The forward implication is clear, that  is if $\la -\phi$ has inverse  $ x  \overset{{(\la-\phi)}^{-1}}{\longmapsto} {(\la -\phi(x))}^{-1}$ in $C_b(\Om, \Bm{\mZ}(E)_\textbf{s})$, then $\la\in\Delta_u$. For the reverse, suppose $\la\in\Delta_u$ and consider the mapping $R: \Om \lo \Bm{\mZ}(E)_\textbf{s}; x \mt R(x):=R(\la,\phi(x) )= (\la -\phi(x))^{-1}$. Then, we see that $R$ is the pointwise inverse of $\la -\phi$  and is bounded. To prove that $R$ is continuous, let $x_\ga$ be a net in $\Om$ converging to $x$ in $\Om$. For all $\ep>0$ and each $z\in E$, using the continuity of $\la -\phi$, we eventually have $||z - (\la -\phi(x_\ga))R(x)z||<\ep$. As such, for each $z\in E$, we obtain
		\beqan
	||R(x_\ga)z - R(x)z||	 =    &   || R(x_\ga)[z - (\la -\phi(x_\ga))R(x)z] || \\
	\le & ||R(x_\ga)||_{\sL(E)}\ep
	\eeqan
	eventually, which can be made small as much as desired, since $R$ is bounded, i.e., $||R||:= \sup_{x\in\Om}||R(x)||_{\sL(E)}<\ty$. Thus, $R\in C_b(\Om, \Bm{\mZ}(E)_\textbf{s})$ which is the desired bounded inverse of $\la -\phi$.
	\eei
\end{proof}

\section{Multiplication semigroup on $ C_0(\Om, E)$ generated by  $(\mM_\phi,  D(\mM_\phi)  )$}\label{sec:msg}

Given a continuous function $\phi : \Om \lo  \Bm{\mZ}(E)_\textbf{s}$, we associate, for each $t\ge 0$, the exponential function \[e^{t\phi}:\Om \lo  \Bm{\mZ}(E)_\textbf{s};\ x \mt e^{t\phi(x)}   \]
which can immediately be seen to be continuous.
Therefore, we  obtain certain multiplication semigroup $(\mT_\phi(t))_{t\ge 0}$  on  $ C_0(\Om, E)$ defined, for each $t\ge 0$, by
\[  s(\cd) \mt \mT_\phi(t)s(\cd):=e^{t\phi(\cd)}s(\cd) \]
which each $\mT_\phi(t)$ is bounded if and only if $e^{t\phi} \in  C_b(\Om, \Bm{\mZ}(E)_\textbf{s})$, using Proposition \ref{prop:mop1}(ii). Indeed, we see that $\mT_\phi(0)=\mI$, and $\mT_\phi(t_1+t_2)= \mT_\phi(t_1)\mT_\phi(t_2)$ for all $t_1,t_2\ge0$. And consequently, $\mT_\phi(t)\in \Bm{\mZ} ( C_0(\Om, E))$ and
\[  ||\mT_\phi(t)|| = ||e^{t\phi}||= \sup_{x\in\Om}||e^{t\phi(x)}||_{\sL(E)} \ \text{ for all } t\ge 0.\]

\begin{rem}\label{rem:msg1}
	We show that the condition $\sup_{x\in\Om}||e^{t_0\phi(x)}||_{\sL(E)}<\ty$ for some $t_0\in(0,1]$ is only what is required for each of the multiplication operator $\mT_\phi(t)$ to be bounded on  $C_0(\Om, E)$, and not necessarily that  $\sup_{x\in\Om}||\phi(x)||_{\sL(E)}<\ty$. To this end, let $q :\Om \lo \C$ be an unbounded continuous function which is bounded from above, i.e., $\sup_{x\in\Om}|q(x)|=\ty$, but  $\sup_{x\in\Om}Req(x)<\ty$. Then the mapping $qI_E : \Om \lo \Bm{\mZ}(E)_\textbf{s}; x \mt q(x)I_E$ is continuous but unbounded since $\sup_{x\in\Om}||q(x)I_E||_{\sL(E)}=   \sup_{x\in\Om}|q(x)|=\ty$. However, the fact that the exponential functions
	\[ e^{tqI_E}: \Om \lo  \Bm{\mZ}(E)_\textbf{s}; \ x\mt  e^{tq(x)I_E }  \]
	are continuous and bounded, since  for each $t\in(0,t_0]$,
	\beqan
\sup_{x\in\Om}||e^{tq(x)I_E}||_{\sL(E)}	 =    &   \sup_{x\in\Om} |e^{tq(x)}|   \\
	= & \sup_{x\in\Om} e^{t Re q(x)}\\
		= & e^{t \sup_{x\in\Om}Re q(x)}<\ty,
	\eeqan
implies that the multiplication operators $\mT_{qI_E}(t)$ are bounded on  $C_0(\Om, E)$, and $ ||\mT_{qI_E}(t)|| = e^{t \sup_{x\in\Om}Re q(x)}$ for each $t\in(0,t_0]$. Hence, in general, given that $||\mT_\phi(t_0)||=\sup_{x\in\Om}||e^{t_0\phi(x)}||_{\sL(E)}<\ty$, then every other $\mT_\phi(t)$ would be bounded since every $t>t_0$ can be written uniquely as $t=nt_0+r$ for some $n\in\N$ with $r\in (0, t_0]$ such that $\mT_\phi(t)=\mT_\phi(t_0)^n \mT_\phi(r)$. 
	
\end{rem}

We now introduce multiplication semigroup induced by (generally) unbounded  multiplication operator $(\mM_\phi,  D(\mM_\phi)  )$. This definition is motivated by \cite[Definition 3.3, p.22]{engel2006short}.

\begin{defn}\label{defn:msg1}
Let  $\phi : \Om \lo  \Bm{\mZ}(E)_\textbf{s}$ be a continuous function such that  \[\sup_{x\in\Om}||e^{t_0\phi(x)}||_{\sL(E)}<\ty \ \text{ for some }\ t_0\in (0,1].  \] Then, we call the semigroup  $(\mT_\phi(t))_{t\ge 0}$ comprised of (bounded) multiplication operators
\[\mT_\phi(t): C_0(\Om, E) \lo C_0(\Om, E);\ s(\cd)\mt  \mT_\phi(t)s(\cd)=e^{t\phi(\cd)}s(\cd)   \]
for all $t\ge0$, the multiplication semigroup generated by the multiplication operator $(\mM_\phi,  D(\mM_\phi)  )$ on $C_0(\Om, E)$.
\end{defn}

The following Proposition justifies the previous definition. These can be considered as generalizations of  \cite[Proposition 3.4, p.23]{engel2006short} and  \cite[Proposition 3.5, p.23-24]{engel2006short}. See also \cite[Lemma 2.8, p.50]{engel2006short}.

\begin{prop}\label{prop:msg1}
	Let $\phi : \Om \lo  \Bm{\mZ}(E)_\textbf{s}$ be a continuous function such that  $\sup_{x\in\Om}||e^{t_0\phi(x)}||_{\sL(E)}<\ty$ for some $t_0\in(0,1]$. Furthermore, let $(\mT_\phi(t))_{t\ge 0}$ be the multiplication semigroup generated by $(\mM_\phi,  D(\mM_\phi)  )$ on $C_0(\Om, E)$ as in Definition \ref{defn:msg1}. Then we have the following.
	
	\bei
	\item[(i)] $(\mT_\phi(t))_{t\ge 0}$ is a $C_0$-semigroup on $C_0(\Om, E)$, i.e., the mapping 
	\[  \R_+ \lo C_0(\Om, E);\ t \mt \mT_\phi(t)s(\cd)  \]
	is continuous  for each $s\in C_0(\Om, E)$; and its (infinitesimal) generator is the multiplication operator $(\mM_\phi,  D(\mM_\phi)  )$ on $C_0(\Om, E)$.
	
	\item[(ii)] The multiplication $C_0$-semigroup $(\mT_\phi(t))_{t\ge 0}$ is uniformly continuous if and only if $\phi\in C_b(\Om, \Bm{\mZ}(E)_\textbf{s})$.
	\eei
\end{prop}
\begin{proof}
	\bei
	\item[(i)] Let $s\in C_0(\Om, E)$ with $||s||\le 1$. And for $\ep>0$, we can choose a compact subset $K\sse \Om$, such that $||s(x)||\le \frac{\ep}{||e^{\phi}|| +1}    $ for all $x\in\Om \bk K$. Moreover, we note that, the fact that the pointwise multiplication operators $(\phi(x)_{x\in\Om})$ induced uniformly continuous $C_0$-semigroup $( e^{t\phi(x)})_{t\ge 0} $ on $E$ generated by $\phi(x)\in\Bm{\mZ}(E)_\textbf{s}$  for each $x\in\Om$, i.e., the mapping 
	\[ \R_+ \lo \Bm{\mZ}(E)_\textbf{s};\ t \mt e^{t\phi(x)}\] is uniformly continuous for each $x\in\Om$, implies that we can find $t_1\in(0, 1]$ such that $||e^{t\phi(x)} - I_E||_{\sL(E)}\le \ep$ for all $x\in K$ and $0\le t\le t_1$. And therefore, we obtain that
	\beqan
	 & ||\mT_\phi(t)s - s||   \\
	 =    &   \sup_{x\in\Om}||e^{t\phi(x)}s(x)  -s(x)       ||    \\
	= & \sup_{x\in\Om} ||( e^{t\phi(x)} -    I_E) s(x) ||           \\
	\le & \sup_{x\in K} ||( e^{t\phi(x)} -    I_E)||_{\sL(E)} || s(x) ||  + \sup_{x\in \Om\bk K}||( e^{t\phi(x)} -    I_E)||_{\sL(E)} || s(x) ||\\
	\le &     \ep \sup_{x\in K}|| s(x) || +   \sup_{x\in \Om\bk K} ||( e^{t\phi(x)} -    I_E)||_{\sL(E)}  \frac{\ep}{||e^{\phi}|| +1} \\
	\le &    \ep +\ep \\
	= & 2\ep
	\eeqan
	for all  $0\le t\le t_1$, showing that  $(\mT_\phi(t))_{t\ge 0}$ is strongly continuous, and hence a $C_0$-semigroup on $C_0(\Om, E)$. \bigskip

	Next, we show that  $(\mM_\phi,  D(\mM_\phi)  )$ is indeed the  (infinitesimal) generator of the  multiplication $C_0$-semigroup $(\mT_\phi(t))_{t\ge 0}$ on  $C_0(\Om, E)$. To this ends, suppose that $(\mM, D(\mM))$ is the generator of the $C_0$-semigroup $(\mT_\phi(t))_{t\ge 0}$ on  $C_0(\Om, E)$.
	
	Let $s\in D(\mM_\phi)$; $x \overset{e^{t\phi}s }{\longmapsto} e^{t\phi(x)}s(x)$ is in $C_0(\Om, E)$, and $\lim_{t\da 0} \frac{ e^{t\phi(x)}s(x) -s(x)}{t} = \phi(x)s(x)= (\mM_\phi s)(x)$ implies that $s\in D(\mM)$; so that  $D(\mM_\phi)\sse D(\mM)$ and $\mM s = \phi\cc s = \mM_\phi s$ for all $s\in D(\mM_\phi)$. \bigskip
	
	And on the other hand, for $s\in C_0(\Om, E)$ the fact that 
	  \[  \lim_{t\da 0} \frac{ e^{t\phi}s-s}{t} \text{ exists}  \iff  x \mt    \lim_{t\da 0} \frac{ e^{t\phi(x)}s(x) -s(x)}{t} = \phi(x)s(x)  \text{ exists }     \]
	implies that $D(\mM)\sse D(\mM_\phi)$ and $\mM_\phi s = \mM s$ for all $s\in D(\mM)$. Hence, $(\mM_\phi,  D(\mM_\phi)  ) = (\mM, D(\mM))$.

	\item[(ii)]  Suppose $\phi\in C_b(\Om, \Bm{\mZ}(E)_\textbf{s})$. Then, see Proposition \ref{prop:mop1}(ii),  $\mM_\phi$ is a bounded multiplication operator and therefore generates a uniformly continuous multiplication semigroup $  (e^{t\mM_\phi} )_{t\ge0}$ on $C_0(\Om, E)$ which must coincides with the semigroup $(\mT_\phi(t))_{t\ge 0}$, since $ e^{t\mM_\phi}s(\cd) = e^{t\phi(\cd)}s(\cd) = \mT_\phi(t)s(\cd)  $ for all  $s\in C_0(\Om, E)$ and $t\ge 0$. 
	
	Now, on the contrary, assume  $\phi : \Om \lo  \Bm{\mZ}(E)_\textbf{s}$ is continuous but unbounded. Let $(x_n)\sse \Om$ be a sequence such that  $0<||\phi(x_n)||_{\sL(E)}\ \forall n\in\N$ and $\lim_{n\ra \ty}||\phi(x_n)||_{\sL(E)}=\ty$. As such, we can find $(z_n)_n\in E$ with $||z_n||=1$ such that $\lim_{n\ra \ty}||\phi(x_n)z_n||=\ty$. Moreover, $||I_E-e^T||>0$ for every $T\in  \Bm{\mZ}(E)_\textbf{s}$ with $||T||_{\sL(E)}=1$, implies there exists $\de>0$ such that $||z_n - e^{t_n\phi(x_n)}z_n||\ge\de \ \forall\ n\in\N$ where $t_n:=1/||\phi(x_n)||$. 
	
	Choose $f_n\in C_0(\Om)$ with $||f_n||=1=f_n(x_n) $ and let $s_n:=f_n\otimes z_n$, i.e., $s_n(\cd)= f_n(\cd)z_n$ so that $s_n\in C_0(\Om, E)$, $s_n(x_n)=z_n$ and $||s_n||=1 \ \forall\ n\in\N$. 
	
	Then, we obtain that 
		\beqan
\de	\le & ||z_n - e^{t_n\phi(x_n)}z_n||   \\
	=   & ||s_n(x_n) -  e^{t_n\phi(x_n)}s_n(x_n)||          \\
	\le &  ||s_n -  e^{t_n\phi}s_n ||         \\
	\le &  ||\mI -  \mT_\phi(t_n)|| 
	\eeqan
$ \forall\ n\in\N$. And since $t_n \ra 0$ in $\R_+$, this implies that  $(\mT_\phi(t))_{t\ge 0}$ is not uniformly continuous.	
	
Equivalently, since the $C_0$-semigroup $(\mT_\phi(t))_{t\ge 0}$ is uniformly continuous if and only if its generator  $(\mM_\phi,  D(\mM_\phi)  )$ is bounded,  Proposition \ref{prop:mop1}(ii) implies that this is the case if and only if  $\phi\in C_b(\Om, \Bm{\mZ}(E)_\textbf{s})$.
	\eei
\end{proof}

\section{Characterization of multiplication semigroups on $ C_0(\Om, E)$ }

 In the following Lemma, we obtain the complete characterization of multiplication $C_0$-semigroups on $ C_0(\Om, E)$ as those arising from continuous functions $\phi : \Om \lo  \Bm{\mZ}(E)_\textbf{s}$ as in Definition \ref{defn:msg1}.  This can been seen as the generalization of   \cite[Proposition 3.6, p.24-25]{engel2006short}. 
 \begin{lemma}\label{lem:msg2}
 	Let  $m_t\in C_b(\Om, \Bm{\mZ}(E)_\textbf{s})$ for each $t\ge0$ such that the induced multiplication semigroup $\mT(t)_{t\ge0}$ given by
 	
 	\[ \mT(t) : C_0(\Om, E) \lo C_0(\Om, E);\ s(\cd)\mt \mT(t)s(\cd):=m_t(\cd)s(\cd)     \]
 	is a $C_0$-semigroup on $C_0(\Om, E)$. Furthermore, suppose that, for each $x\in\Om$, the induced pointwise multiplication semigroup $T_x(t)_{t\ge0}$ given by
 	\[T_x(t): E \lo E;\ z\mt T_x(t)z:=m_t(x)z    \]
 	is uniformly continuous on $E$. Then, there exists a continuous function $\phi : \Om \lo  \Bm{\mZ}(E)_\textbf{s}$, where $(T_x(t))_{t\ge0}= (e^{t\phi(x)})_{t\ge0}$ for each $x\in\Om$ with
 	\[\sup_{x\in\Om}||e^{t_0\phi(x)}||_{\sL(E)}<\ty \ \text{ for some }\ t_0\in (0,1]  \]
 	such that $m_t(\cd)= e^{t\phi(\cd)}$ for all $t\ge0$.
 \end{lemma}

\begin{proof}

By the strong continuity of $\mT(t)_{t\ge0}$ on $C_0(\Om, E)$, i.e., the continuity of the mapping
\[  \R_+ \lo C_0(\Om, E);\ t \mt m_t(\cd)s(\cd)  \]
 for each $s\in C_0(\Om, E)$; we obtain that, the pointwise multiplication semigroup $T_x(t)_{t\ge0}$ on $E$ given by
 \[ T_x(t) : E \lo E; \ z \mt m_t(x)z      \]
is a $C_0$-semigroup on $E$ for each $x\in\Om$. 
In particular, given that each $T_x(t)_{t\ge0}$  is uniformly continuous implies that it is generated by a bounded multiplication operator, say $\phi(x)\in\Bm{\mZ}(E)_\textbf{s}$ for each $x\in\Om$. Indeed, we would have that $\phi(x):=\lim_{t\da 0} \frac{T_x(t) -I_E}{t} \in\sL(E)$ in the operator norm  such that $T_x(t)= e^{t\phi(x)}$  for all $t\ge0$ and each $x\in\Om$; and since $\Bm{\mZ}(E)$ is norm-closed in $\sL(E)$, it follows that $\phi(x)\in \Bm{\mZ}(E)_\textbf{s}$ for all $x\in\Om.$ See also \cite[Proposition 5.15, p.288]{arendt1986}.

Now, since $\mT(t)s(\cd) = T_{(\cd)}(t)s(\cd)$ for all $t\ge0$ and $s\in C_0(\Om, E)$, it definitely must be that $||T(t)||=\sup_{x\in\Om}||T_x(t)||_{\sL(E)} = \sup_{x\in\Om}||e^{t\phi(x)}||_{\sL(E)}<\ty$ for all $t\ge0$. Therefore, it does follows that $m_t(x)= e^{t\phi(x)}$ for all $x\in\Om$ and $t\ge0.$ \bigskip

So, it remains only to show that the mapping $\phi : \Om \lo \Bm{\mZ}(E)_\textbf{s};\ x\mapsto\phi(x)$ is continuous. And, on the contrary, assume $\phi$ is not continuous. Let $(x_n)\sse\Om$ be a sequence converging to $x\in\Om$ but $\phi(x_n) \nrightarrow \phi(x)$ in $\Bm{\mZ}(E)_\textbf{s}$. As such, for $\ep>0$ we can find $z_0\in E$ with $||z_0||=1$ such that $  ||\phi(x)z_0 - \phi(x_n)z_0||\ge\ep$ $ \forall\ n\in\N$. Since
\[\phi(x)z_0  = \lim_{t\da 0} \frac{T_x(t)z_0 -z_0}{t} \ \text{ and } \phi(x_n)z_0= \lim_{t\da 0} \frac{T_{x_n}(t)z_0 -z_0}{t}\]
  $ \forall\ n\in\N$, we can find some $t_1\in(0,1]$ such that 
	\beqan
\ep	\le & ||\phi(x)z_0 - \phi(x_n)z_0||   \\
\le   & ||  \frac{T_x(t)}{t}z_0 - \frac{T_{x_n}(t)}{t}z_0 ||
\eeqan
$ \forall\ n\in\N$  with $0<t\le t_1$, which is contradicting the fact that $m_t(\cd)=T_{(\cd)}(t)$ are continuous. Hence,  $\phi : \Om \lo \Bm{\mZ}(E)_\textbf{s}$ is continuous.

\end{proof}

\paragraph{Proof of Theorem \ref{thm1}:} Combining the above Lemma \ref{lem:msg2}  with Proposition \ref{prop:mop1}(i-iii) and Proposition \ref{prop:msg1} proves the assertions of Theorem \ref{thm1}.

\section{Examples}\label{sec:exm}

In this last section, we state some observations and present concrete examples of multiplication semigroups. This can be seen as extension of \cite[Example 3.7, p.25]{engel2006short} in the scalar-valued case.

	\ben 
	\item[(i)] If $K$ is a compact (Hausdorff) space and $E$ is a Banach lattice, then every continuous function $\phi : K \lo  \Bm{\mZ}(E)_\textbf{s}$ is already bounded. Therefore, the induced multiplication operator $\mM_\phi$ on $C(K, E)$  is bounded, and hence every generated multiplication semigroup $(\mT_\phi(t))_{t\ge 0}$ is always uniformly continuous on $C(K, E)$.
	
	\item[(ii)] Let $q: \Om \lo \C$ be an unbounded continuous function which is bounded above as in Remark \ref{rem:msg1}, set $w:=\sup_{x\in\Om} Re q(x)$ and $\phi:=qI_E$. Then, the induced unbounded multiplication operator $(\mM_\phi, D(\mM_\phi))  $ generates a strongly continuous multiplication semigroup $(\mT_\phi(t))_{t\ge 0}$ on $C_0(\Om, E) $ given by
	\[   \mT_\phi(t) s(\cd) = e^{tq(\cd)}s(\cd) \]
	for all $t\ge0$. Moreover, $\sig(\mM_\phi) = \overline{q(\Om)}$ and $||\mT_\phi(t)|| = e^{wt}$	for all $t\ge0$.
	\item[(iii)] Let $\Om:=\N$ and $E:=\C^2$, then each sequence $ (\phi^1,\phi^2):= \text{diag} (\phi^1_n, \phi^2_n)_{n\in\N}$ of 2 by 2 diagonal complex matrices induced a multiplication operator $(\mM_{(\phi^1,\phi^2)}, D(\mM_{(\phi^1,\phi^2)} )   ) $
	\[ (s^1_n, s^2_n)_{n\in\N} \mt (\phi^1_ns^1_n,\phi^2_ns^2_n  )_{n\in\N}    \]
	on $C_0(\N, \C^2)=:c_0(\C^2)$. For $(\phi^1,\phi^2):= \text{diag} (in, -n^2)_{n\in\N}$, the induced multiplication operator $(\mM_{(\phi^1,\phi^2)}, D(\mM_{(\phi^1,\phi^2)} )   ) $  generates a strongly continuous multiplication semigroup $(\mT_{(\phi^1,\phi^2) }(t) )_{t\ge 0}$ on $c_0(\C^2)$ given by
	\[ \mT_{(\phi^1,\phi^2) }(t)(s^1_n, s^2_n)_{n\in\N} =   (e^{int}s^1_n, e^{-n^2t}s^2_n )_{n\in\N}\ \text{ for all } \ t\ge0. \]
	
	\item[(iv)] As a typical instance of (i) above, let $K:= \lb 1, 2, \hdots, m \rb$ be a finite set and $E$ a Banach lattice; so that $C(K, E)= E^m$ is an m-copies of $E$. Given an m-ordinates $\phi:=(\phi_1, \hdots \phi_m  )$ of multiplication operators on $E$, i.e., $\phi_j\in\Bm{\mZ}(E)_\textbf{s}$ for $1\le j\le m$, the induced (bounded) multiplication operator
		\[ (s_1,\hdots, s_m ) \mt (\phi_1s_1,\hdots, \phi_m s_m    )   \]
	on $E^m$ corresponds to the diagonal operator matrix $A_\phi = \text{diag} ( \phi_1, \hdots \phi_m ) $. And, therefore, the generated uniformly continuous multiplication semigroup $(e^{tA_\phi})_{t\ge 0} $ on $E^m$ is given by diagonal operator matrices	
	\[e^{tA_\phi} =  \text{diag}(e^{t\phi_1},\hdots, e^{t\phi_m} ) \ \text{ for all } \ t\ge0.  \]\bigskip
	Moreover, $\sig(A_\phi) = \underset{1\le j\le m}{\bigcup}\sig(\phi_j)$ and $||e^{tA_\phi}|| = \underset{1\le j\le m}{\max}  || e^{t\phi_j}|| \text{ for all } t\ge0.$ 
	\een

\bibliographystyle{siam}

\bibliography{references}

\begin{thebibliography}{10}

\bibitem{arendt1986}
{\sc W.~Arendt, A.~Grabosch, G.~Greiner, U.~Moustakas, R.~Nagel,
  U.~Schlotterbeck, U.~Groh, H.~P. Lotz, and F.~Neubrander}, {\em One-parameter
  semigroups of positive operators}, vol.~1184, Springer, 1986.

\bibitem{ChristianRetha22}
{\sc C.~Budde and R.~Heymann}, {\em Extrapolation of operator-valued
  multiplication operators}, Quaest. Math., 45 (2022), pp.~347--356.

\bibitem{engel2006short}
{\sc K.-J. Engel and R.~Nagel}, {\em A short course on operator semigroups},
  Springer Science \& Business Media, 2006.

\bibitem{ErcanWickstead96}
{\sc Z.~Ercan and A.~W. Wickstead}, {\em Banach lattices of continuous {B}anach
  lattice-valued functions}, J. Math. Anal. Appl., 198 (1996), pp.~121--136.

\bibitem{Gierz1982}
{\sc G.~Gierz}, {\em Bundles of topological vector spaces and their duality},
  vol.~57 of Queen's Papers in Pure and Applied Mathematics, Springer-Verlag,
  Berlin-New York, 1982.
\newblock With an appendix by the author and Klaus Keimel.

\bibitem{Anna2018}
{\sc A.~Goli\'{n}ska}, {\em Semigroups of {H}adamard multipliers on the space
  of real analytic functions}, Ann. Polon. Math., 121 (2018), pp.~217--229.

\bibitem{Graser97}
{\sc T.~Graser}, {\em Operator multipliers generating strongly continuous
  semigroups}, Semigroup Forum, 55 (1997), pp.~68--79.

\bibitem{Heymann2013}
{\sc R.~Heymann}, {\em Eigenvalues and stability properties of multiplication
  operators and multiplication semigroups}, Mathematische Nachrichten, 287
  (2013), pp.~574 -- 584.

\bibitem{Holderrieth91}
{\sc A.~Holderrieth}, {\em Matrix multiplication operators generating one
  parameter semigroups}, Semigroup Forum, 42 (1991), pp.~155--166.

\bibitem{hudzik2014multiplication}
{\sc H.~Hudzik, R.~Kumar, and H.~Sani}, {\em Multiplication semigroups on
  {B}anach function spaces}, arXiv preprint arXiv:1402.3914,  (2014).

\bibitem{MerveSerkanEmrah2020}
{\sc M.~\.{I}lkhan, S.~Demiriz, and E.~E. Kara}, {\em Multiplication operators
  on {C}es\`aro second order function spaces}, Positivity, 24 (2020),
  pp.~605--614.

\bibitem{David2022}
{\sc T.~D. Olabiyi}, {\em Positive weighted {K}oopman semigroups on {B}anach
  lattice modules}, Master's thesis. Stellenbosch University,  (2022).

\bibitem{TodorovTurowska2010}
{\sc I.~G. Todorov and L.~Turowska}, {\em Schur and operator multipliers}, in
  Banach algebras 2009, vol.~91 of Banach Center Publ., Polish Acad. Sci. Inst.
  Math., Warsaw, 2010, pp.~385--410.

\bibitem{van1993abstract}
{\sc J.~Van~Neerven}, {\em Abstract multiplication semigroups}, Math. Z, 213
  (1993), pp.~1--15.

\bibitem{wickstead17}
{\sc A.~W. Wickstead}, {\em Banach lattice algebras: some questions, but very
  few answers}, Positivity, 21 (2017), pp.~803--815.

\end{thebibliography}

\section*{About the author:}
Tobi David Olabiyi.
Stellenbosch University, Mathematics Division, Faculty of Science, Merriman Avenue, 7600 Stellenbosch, South Africa. 25175645@sun.ac.za\\/davidtobiolabiyi@gmail.com

\end{document}